\documentclass{article}

\usepackage[centertags]{amsmath}
\usepackage{hyperref}
\usepackage{amsfonts}
\usepackage{amssymb}
\usepackage{amsthm}
\usepackage{newlfont}
\usepackage{amscd}
\usepackage{amsmath,amscd}
\usepackage{graphicx}
\usepackage[all]{xy}
\usepackage{verbatim}

\usepackage{color}

\newcommand{\barr}{\overline}

\newcommand{\NN}{\mathbb{N}}
\newcommand{\cA}{\mathcal{A}}
\newcommand{\cB}{\mathcal{B}}
\newcommand{\cC}{\mathcal{C}}
\newcommand{\cD}{\mathcal{D}}

\newcommand{\cF}{\mathcal{F}}

\newcommand{\cP}{\mathcal{P}}

\newcommand{\cW}{\mathcal{W}}

\newtheorem{thm}{Theorem}[section]

\newtheorem{cor}[thm]{Corollary}

\newtheorem{lem}[thm]{Lemma}
\newtheorem{prop}[thm]{Proposition}

\newtheorem{example}{Example}
\theoremstyle{definition}
\newtheorem{define}[thm]{Definition}
\theoremstyle{remark}
\newtheorem{rem}[thm]{Remark}

\DeclareMathOperator{\Pro}{Pro}

\DeclareMathOperator{\Hom}{Hom}

\DeclareMathOperator{\precolim}{colim}

\def\colim{\mathop{\precolim}}

\def \mcal{\mathcal}

\DeclareFontEncoding{OT2}{}{} 
\DeclareTextFontCommand{\textcyr}{\fontencoding{OT2}\fontfamily{wncyr}\fontseries{m}\fontshape{n}\selectfont}

\begin{document}
\title{Functorial Factorizations in Pro Categories}

\author{Ilan Barnea \and Tomer M. Schlank}

\maketitle

\begin{abstract}
In this paper we prove a few propositions concerning factorizations of morphisms in pro categories, the most important of which solves an open problem of Isaksen \cite{Isa} concerning the existence of certain types of functorial factorizations. "On our way" we explain and correct an error in one of the standard references on pro categories.
\end{abstract}

\tableofcontents

\section{Introduction}
Pro-categories introduced by Grothendieck \cite{SGA4-I} have found many applications over the years in fields such as algebraic geometry \cite{AM}, shape theory \cite{MS} and more. In this paper we prove a few propositions concerning factorizations of morphisms in pro categories. These will later be used to deduce certain facts about model structures on pro categories. The most important conclusion of this paper will be solving an open problem of Isaksen \cite{Isa} concerning the existence of functorial factorizations in what is known as the strict model structure on a pro category. In order to state our results more accurately we give some definitions in a rather brief way. For a more detailed account see section ~\ref{s:prelim}

Let $\mcal{C}$ be a category and $M$ a class of morphisms in $\cC$. We denote by:
\begin{enumerate}
\item $R(M)$ the class of morphisms in $\mcal{C}$ that are retracts of morphisms in $M$.
\item $^{\perp}M$ the class of morphisms in $\cC$ having the left lifting property w.r.t. all maps in $M$.
\item $M^{\perp}$ the class of morphisms in $\cC$ having the right lifting property w.r.t. all maps in $M$.
\end{enumerate}

Let $N,M$ be classes of morphisms in $\cC$. We will say that there exist a factorization in $\cC$ into a morphism in $N$ followed by a morphism in $M$ (and denote $Mor(\mcal{C}) = M\circ N$) if every map $X\to Y $ in $\mcal{C}$ can be factored as $X\xrightarrow{q} L\xrightarrow{p} Y $ s.t. $q$ is in ${N}$ and $p$ is in $M$. The pair $(N,M)$ will be called a weak factorization system in $\cC$ (see \cite{Rie}) if the following holds:
\begin{enumerate}
\item $Mor(\mcal{C}) = M\circ N$.
\item $N=^{\perp}M$.
\item $N^{\perp}=M$.
\end{enumerate}

A functorial factorization in $\cC$,  is a functor:
$\cC^{\Delta^1}\to\cC^{\Delta^2}$ denoted:
$$(X\xrightarrow{f}Y)\mapsto ({X}\xrightarrow{q_f} L_{f}\xrightarrow{p_f}Y)$$
s.t.
\begin{enumerate}
\item For any morphism $f$ in $\cD$ we have: $f=p_f\circ q_f$.
\item For any morphism:
$$\xymatrix{{X}\ar[r]^{f}\ar[d]_{l} & {Y}\ar[d]^{k}\\
            {Z} \ar[r]^t   &   {W},}$$
in $\cD^{\Delta^1}$ the corresponding morphism in $\cD^{\Delta^2}$ is of the form:
$$\xymatrix{{X}\ar[r]^{q_f}\ar[d]_{l} & L_{f}\ar[r]^{h_f}\ar[d]^{L_{(l,k)}} & {Y}\ar[d]^{k}\\
            {Z} \ar[r]^{q_t} & L_t\ar[r]^{p_t}    &   {W}.}$$
\end{enumerate}

The above functorial factorization is said to be into a morphism in $N$ followed be a morphism in $M$ if for every $f\in Mor(\cC)$ we have $q_f\in N,p_f\in M$.

We will denote $Mor(\mcal{C}) = ^{func}M\circ N$ if there exist a \emph{functorial} factorization in $\cC$ into a morphism in $N$ followed by a morphism in $M$. The pair $(N,M)$ will be called a \emph{functorial} weak factorization system in $\cC$ if the following holds:
\begin{enumerate}
\item $Mor(\mcal{C}) =^{func} M\circ N$.
\item $N=^{\perp}M$.
\item $N^{\perp}=M$.
\end{enumerate}

Note that $Mor(\mcal{C}) =^{func} M\circ N$ clearly implies $Mor(\mcal{C}) = M\circ N$.

The category $\Pro(\mcal{C})$ has as objects all diagrams in $\cC$ of the form $I\to \cC$ s.t. $I$ is small and directed (see Definition ~\ref{d_directed}). The morphisms are defined by the formula:
$$\Hom_{\Pro(\mcal{C})}(X,Y):=\lim\limits_s \colim\limits_t \Hom_{\mcal{C}}(X_t,Y_s).$$
Composition of morphisms is defined in the obvious way.

Note that not every map in $Pro(\cC)$ is a natural transformation (the source and target need not even have the same indexing category). However, every natural transformation between objects in $Pro(\cC)$ having the same indexing category, induces a morphism in $Pro(\cC)$ between these objects, in a rather obvious way.

Let $M$ be a class of morphisms in $\cC$. We denote by $Lw^{\cong}(M)$ the class of morphisms in $\Pro(\mcal{C})$ that are \textbf{isomorphic} to a morphism that comes from a natural transformation which is a level-wise $M$-map.

If $T$ is a partially ordered set, then we view $T$ as a category which has a single morphism $u\to v$ iff $u\geq v$. A cofinite poset is a poset $T$ s.t. for every $x \in T$ the set $T_x:=\{z\in T| z \leq x\}$ is finite.

Suppose now that $\mcal{C}$ has finite limits. Let $T$ a small cofinite poset and $F:X\to Y$ a morphism in $\mcal{C}^T$. Then $F$ will be called a special $M$-map, if the natural map $X_t \to Y_t \times_{\lim \limits_{s<t} Y_s} \lim \limits_{s<t} X_s $ is in $M$, for every $t\in T$. We denote by $Sp^{\cong}(M)$ the class of morphisms in $\Pro(\mcal{C})$ that are \textbf{isomorphic} to a morphism that comes from a (natural transformation which is a) special $M$-map.

The following proposition gives strong motivation for the above defined concepts. It is proved in sections ~\ref{ss_lift} and ~\ref{ss_fact}.

\begin{prop}\label{p:fact}
Let $\mcal{C}$ be a category that has finite limits, and let $N,M$ be classes of morphisms in $\cC$. Then:
\begin{enumerate}
\item $^{\perp}R(Sp^{\cong}(M))=^{\perp}Sp^{\cong}(M)=^{\perp}M.$
\item If $Mor(\mcal{C}) = M\circ N$ then $Mor(Pro(\mcal{C})) = Sp^{\cong}(M)\circ Lw^{\cong}(N)$.
\item If $Mor(\mcal{C}) = M\circ N$ and $N\perp M$ (in particular, if $(N,M)$ is a weak factorization system in $\cC$), then $(Lw^{\cong}(N),R(Sp^{\cong}(M)))$ is a weak factorization system in $Pro(\cC)$.
\end{enumerate}
\end{prop}

In the proof Proposition ~\ref{p:fact} part (2) we use the classical theorem saying that for every small directed category $I$ there exist a cofinite directed set $A$ and a cofinal functor: $p:A\to I.$ In \cite{Isa}, Isaksen gives two references to this theorem. one is \cite{EH} Theorem 2.1.6 and the other is \cite{SGA4-I} Proposition 8.1.6. We take this opportunity to explain and correct a slight error in the proof given in \cite{EH}.

The proof of Proposition ~\ref{p:fact} is strongly based on \cite{Isa} sections 4 and 5, and most of the ideas can be found there. The main novelty in this paper is the following theorem, proved in Section ~\ref{s_ff}:
\begin{thm}\label{t_main}
Let $\mcal{C}$ be a category that has finite limits, and let $N,M$ be classes of morphisms in $\cC$. Then:
\begin{enumerate}
\item If $Mor(\mcal{C}) =^{func} M\circ N$ then $Mor(Pro(\mcal{C})) =^{func}  Sp^{\cong}(M)\circ Lw^{\cong}(N)$.
\item If $Mor(\mcal{C}) = ^{func}M\circ N$ and $N\perp M$ (in particular, if $(N,M)$ is a functorial weak factorization system in $\cC$), then $(Lw^{\cong}(N),R(Sp^{\cong}(M)))$ is a functorial weak factorization system in $Pro(\cC)$.
\end{enumerate}
\end{thm}

The factorizations constructed in the proof of Proposition ~\ref{p:fact} and Theorem ~\ref{t_main} both use Reedy type factorizations (see Section ~\ref{ss_reedy}). These are precisely the factorizations constructed by Edwards and Hastings in \cite{EH} and by Isaksen in \cite{Isa}. The main novelty here is that we show that these factorizations can be made functorial (given a functorial factorization in the original category). Our main tool in proving this will be defining a category equivalent to $Pro(\cC)$, which we call $\barr{Pro}(\cC)$. This category can be thought of as another model for $Pro(\cC)$, and we believe it might also be convenient for other applications. We now describe briefly the category $\barr{Pro}(\cC)$, for more details see Section ~\ref{s:barr}.

Let $A$ be a cofinite directed set. We will say that $A$ has infinite hight if for every $a\in A$ there exist $a'\in A$ s.t. $a<a'$. An object in $\barr{Pro}(\cC)$ is a diagram $F:A\to \cC$, s.t. $A$ is a cofinite directed set of infinite hight. If $F:A\to \cC,G:B\to\cC$ are objects in $\barr{Pro}(\cC)$, A \emph{pre morphism} from $F$ to $G$ is a defined to be a pair $(\alpha,\phi)$, s.t. $\alpha:B\to A$ is a strictly increasing function, and $\phi:\alpha^*F=F\circ\alpha\to G$ is a natural transformation.

We define a partial order on the set of pre morphisms from $F$ to $G$ by setting $(\alpha',\phi')\geq(\alpha,\phi)$ iff for every $b\in B$ we have $\alpha'(b)\geq\alpha(b)$, and the following diagram commutes:
$$
\xymatrix{ & F(\alpha(b)) \ar[dr]^{\phi_b} & \\
F(\alpha'(b))\ar[ur]\ar[rr]^{\phi'_b} & & G(b),}
$$
(the arrow $F(\alpha'(b))\to F(\alpha(b))$ is of course the one induced by the unique morphism $\alpha'(b)\to \alpha(b)$ in $A$).

For $F,G \in \barr{Pro}(\cC)$ we denote by $P(F,G)$ the poset of pre-morphisms from $F$ to $G$.
We now define a \emph{morphism} from $F$ to $G$ in $\barr{Pro}(\cC)$ to be a connected component of $P(F,G)$. We will show (see Corollary ~\ref{c_directed}) that every such connected component is a directed poset. Composition in $\barr{Pro}(\cC)$ is defined by the formula:
$$(\beta,\psi)\circ(\alpha,\phi)=(\alpha\circ\beta,\psi\circ\phi_{\beta})$$.

We construct a natural functor $i:\barr{Pro}(\cC)\to Pro(\cC)$ (the object function of this functor being the obvious one). We show that $i$ is a subcategory inclusion, that is essentially surjective. It follows that $i$ is a categorical equivalence.

When working with pro-categories, it is frequently useful to have some kind of homotopy theory of pro-objects. Model categories, introduced in \cite{Qui}, provide a very general context in which it is possible to set up the basic machinery of homotopy theory. Given a category $\cC$, it is thus desirable to find conditions on $\cC$ under which $Pro(\cC)$ can be given a model structure. It is natural to begin with assuming that $\cC$ itself has a model structure, and look for a model structure on $Pro(\cC)$ which is in some sense induced by that of $\cC$. The following definition is based on the work of Edwards and Hastings \cite{EH}, Isaksen \cite{Isa} and others:

\begin{define}
Let $(\cC,\cW,\cF,\cC of)$ be a model category. The strict model structure on $Pro(\cC)$ (if it exists) is defined by letting the acyclic cofibrations be $^{\perp}\cF$ and the cofibrations be  $^{\perp}(\cW\cap\cF)$.
\end{define}

This model structure is called the strict model structure on $Pro(\cC)$ because several other model structures on the same category can be constructed from it through localization (which enlarges the class weak equivalences).

From Proposition ~\ref{p:fact} it clearly follows that in the strict model structure, if it exists, the cofibrations are given by $Lw^{\cong}(\cC of)$, the acyclic cofibrations are given by $Lw^{\cong}(\cW\cap\cC of)$, the fibrations are given by $R(Sp^{\cong}(\cF))$ and the acyclic fibrations are given by $R(Sp^{\cong}(\cF\cap\cW))$. The weak equivalences can then be characterized as maps that can be decomposed into an acyclic cofibration followed be an acyclic fibration.

Edwards and Hastings, in \cite{EH}, give sufficient conditions on a model category $\cC$ for the strict model structure on $Pro(\cC)$ to exist. Isaksen, in \cite{Isa}, gives different sufficiant conditions on $\cC$ and also shows that under these conditions the weak equivalences in the strict model structure on $Pro(\cC)$ are given by $Lw^{\cong}(\cW)$.

\begin{rem}
It should be noted that we are currently unaware of any example of a model category $\cC$ for which one can show that the strict model structure on $Pro(\cC)$ does not exist.
\end{rem}

The existence of the strict model structure implies that every map in $Pro(\cC)$ can be factored into a (strict) cofibration followed by a (strict) trivial fibration, and into a (strict) trivial cofibration followed by a (strict) fibration. However, the existence of functorial factorizations of this form was not shown, and remained an open problem (see \cite{Isa}  Remark 4.10 and \cite{Cho}). The existence of functorial factorizations in a model structure is important for many constructions (such as framing, derived functor (between the model categories themselves) and more). In more modern treatments of model categories (such as \cite{Hov} or \cite{Hir}) it is even part of the axioms for a model structure.

From Theorem ~\ref{t_main} it clearly follows that if $\cC$ is a model category in the sense of \cite{Hov} or \cite{Hir}, that is, a model category with functorial factorizations, and if the strict model structure on $Pro(\cC)$ exists, then the model structure on $Pro(\cC)$ also admits functorial factorizations.

\section{Preliminaries on Pro-Categories}\label{s:prelim}

In this section we bring a short review of the necessary background on pro-categories. Some of the definitions and lemmas given here are slightly non standard. For more details we refer the reader to \cite{AM}, \cite{EH}, and \cite{Isa}.

\begin{define}\label{d_directed}
A category $I$ is called \emph{cofiltered} (or directed) if the following conditions are satisfied:
\begin{enumerate}
\item $I$ is non-empty.
\item for every pair of objects $s,t \in I$, there exists an object $u\in I$, together with
morphisms $u\to s$ and $u\to t$.
\item for every pair of morphisms $f,g:s\to t$ in $I$, there exists a morphism $h:u\to s$ in $I$, s.t. $f\circ h=g\circ h$.
\end{enumerate}
\end{define}

If $T$ is a partially ordered set, then we view $T$ as a category which has a single morphism $u\to v$ iff $u\geq v$. Note that this convention is opposite from the one used by some authors. Thus, a poset $T$ is directed iff $T$ is non-empty, and for every $a,b\in T$, there exist $c\in T$, s.t. $c\geq a,c\geq b$. In the following, instead of saying "a directed poset" we will just say "a directed set".

\begin{define}\label{def CDS}
A cofinite poset is a poset $T$ s.t. for every $x \in T$ the set $T_x:=\{z\in T| z \leq x\}$ is finite.
\end{define}

\begin{define}\label{def_deg}
Let $A$ be a cofinite poset. We define the degree function of $A$: $d=d_A:A\to\NN$, by:
$$d(a):=max\{n\in \NN|\exists a_0<...<a_n=a\}.$$
For every $n\geq -1$ we define: $A^n:=\{a\in A|d(a)\leq n\}$ $(A^{-1}=\phi)$.
\end{define}
Thus $d:A\to\NN$ is a strictly increasing function. The degree function enables us to define or prove things concerning $A$ inductively, since clearly: $A=\cup_{n\geq 0}A^n$. Many times in this paper, when defining (or proving) something inductively, we will skip the base stage. This is because we begin the induction from $n=-1$, and since $A^{-1}=\phi$ there is nothing to define (or prove) in this stage. the sceptic reader can check carefully the first inductive step to see that this is justified.

We shall use repeatedly the following notion:
\begin{define}\label{d:Reysha}
Let $T$ be a partially ordered set, and let $A$ be a subset of $T$. We will say that $A$ is a Reysha of $T$, if $x\in A,y\in T, y <x$, implies: $y\in A$.
\end{define}

\begin{example}
$T$ is a Reysha of $T$. If $t\in T$ is a maximal element, then $T \backslash \{t\}$ is a Reysha of $T$. For any $t\in T$: $T_t$ (see ~\ref{def CDS}) is a Reysha of $T$.
\end{example}

\begin{define}
Let $C$ be a category. The category $C^{\lhd}$ has as objects: $Ob(C)\coprod{\infty}$, and the morphisms are the morphisms in $C$, together with a unique morphism: $\infty\to c$, for every $c\in C$.
\end{define}

In particular, if $C=\phi$ then $C^{\lhd}=\{\infty\}$.

Note that if $A$ is a cofinite poset, $a\in A$  and  $n=d(a)$ then $A_a$ is naturally isomorphic to $(A_a^{n-1})^{\lhd}$ (where $A_a^{n-1}$ is just $(A_a)^{n-1}$).

\begin{lem}\label{l:eqiv_directed}
A cofinite poset $A$ is directed iff for every finite Reysha $R \subset A$ (see Definition ~\ref{d:Reysha}),
there exist an element $c \in A$ such that $c\geq r$, for every $r\in R$.
A category $\cC$ is directed iff for every finite poset $R$, and for every functor $F:R\to \cC$, there exist $c\in \cC$, together with compatible morphisms $c\to F(r)$, for every $r\in R$ (that is, a morphism $Diag(c)\to F$ in $\cC^R$, or equivalently we can extend the functor $F:R\to T$ to a functor $R^{\lhd}\to \cC$).
\end{lem}

\begin{proof}
Clear.
\end{proof}
A category is called small if it has a small set of objects and a small set of morphisms

\begin{define}\label{def_pro}
Let $\mcal{C}$ be a category. The category $\Pro(\mcal{C})$ has as objects all diagrams in $\cC$ of the form $I\to \cC$ s.t. $I$ is small and directed (see Definition ~\ref{d_directed}). The morphisms are defined by the formula:
$$\Hom_{\Pro(\mcal{C})}(X,Y):=\lim\limits_s \colim\limits_t \Hom_{\mcal{C}}(X_t,Y_s).$$
Composition of morphisms is defined in the obvious way.
\end{define}

Thus, if $X:I\to \mcal{C},Y:J\to \mcal{C}$ are objects in $\Pro(\mcal{C})$, giving a morphism $X\to Y$ means specifying, for every $s\in J$ a morphism $X_t\to Y_s$ in $\mcal{C}$, for some $t\in I$. These morphisms should of course satisfy some compatibility condition. In particular, if the indexing categories are equal: $I=J$, then any natural transformation: $X\to Y$ gives rise to a morphism $X\to Y$ in $\Pro(C)$. More generally, if $\alpha:J\to I$ is a functor, and $\phi:\alpha^*X\to Y$ is a natural transformation, then the pair $(\alpha,\phi)$ determines a morphism $X\to Y$ in $\Pro(C)$ (for every $s\in J$ we take the morphism $\phi_s:X_{\alpha(s)}\to Y_s$).

Let $f:X\to Y$ be a morphism in $Pro(\cC)$. A morphism in $\cC$ of the form $X_r\to Y_s$, that represents the $s$ coordinate of $f$ in $\colim\limits_{t\in I} \Hom_{\mcal{C}}(X_t,Y_s)$ will be called "representing $f$".

The word pro-object refers to objects of pro-categories. A \emph{simple} pro-object
is one indexed by the category with one object and one (identity) map. Note that for any category $\mcal{C}$, $\Pro(\mcal{C})$ contains $\mcal{C}$ as the full subcategory spanned by the simple objects.

\begin{define}\label{def natural}
Let $\mcal{C}$ be a category with finite limits, $M \subseteq Mor(\mcal{C})$ a class of morphisms in $\mcal{C}$, $I$ a small category and $F:X\to Y$ a morphism in $\mcal{C}^I$. Then $F$ will be called:
\begin{enumerate}
\item A levelwise $M$-map, if for every $i\in I$: the morphism $X_i\to Y_i$ is in $M$. We will denote this by $F\in Lw(M)$.
\item A special $M$-map, if the following holds:
    \begin{enumerate}
    \item The indexing category $I$ is a cofinite poset (see Definition ~\ref{def CDS}).
    \item The natural map $X_t \to Y_t \times_{\lim \limits_{s<t} Y_s} \lim \limits_{s<t} X_s $ is in $M$, for every $t\in I$.
    \end{enumerate}
    We will denote this by $F\in Sp(M)$.
\end{enumerate}
\end{define}

Let $\mcal{C}$ be a category. Given two morphisms $f,g \in Mor(\cC)$ we denote by $f \perp g$ to say
that $f$ has the left lifting property w.r.t $g$.
If $M,N\subseteq Mor(C)$, we denote by $M \perp N$ to say that $f\perp g$ for every $f \in M, g\in N$.

\begin{define}\label{def mor}
Let $\mcal{C}$ be a category with finite limits, and $M \subseteq Mor(\mcal{C})$ a class of morphisms in $\mcal{C}$. Denote by:
\begin{enumerate}
\item $R(M)$ the class of morphisms in $\mcal{C}$ that are retracts of morphisms in $M$. Note that $R(R(M))=R(M)$
\item ${}^{\perp}M$ the class of morphisms in $\mcal{C}$ having the left lifting property w.r.t. any morphism in $M$.
\item $M^{\perp}$ the class of morphisms in $\mcal{C}$ having the right lifting property w.r.t. any morphism in $M$.
\item $Lw^{\cong}(M)$ the class of morphisms in $\Pro(\mcal{C})$ that are \textbf{isomorphic} to a morphism that comes from a natural transformation which is a level-wise $M$-map.
\item $Sp^{\cong}(M)$ the class of morphisms in $\Pro(\mcal{C})$ that are \textbf{isomorphic} to a morphism that comes from a natural transformation which is a special $M$-map.
\end{enumerate}
\end{define}

Note that:
$$
(M \subset  {}^\perp N) \Leftrightarrow (N \subset   M^\perp) \Leftrightarrow (M \perp N).
$$

\begin{lem}\label{l:ret_lw}
Let $M$ be any class of morphisms in $\mcal{C}$. Then $$R(Lw^{\cong}(M)) = Lw^{\cong}(M).$$
\end{lem}
\begin{proof}
See \cite{Isa}, Proposition 2.2.
\end{proof}

\begin{lem}\label{c:ret_lift}
Let $M$ be any class of morphisms in $\mcal{C}$. Then:
$$(R(M))^{\perp} = M^{\perp},\; {}^{\perp}(R(M)) = {}^{\perp}M,$$
$$R(M^{\perp}) = M^{\perp},\; R({}^{\perp}M) = {}^{\perp}M.$$
\end{lem}
\begin{proof}
Easy diagram chase.
\end{proof}

\section{Factorizations in pro categories}\label{s_fact}

The main purpose of this section is to prove Proposition ~\ref{p:fact}. It is done in Lemma ~\ref{l:SpMo_is_Mo} and Propositions ~\ref{p_f} and ~\ref{p_wfs}. We also explain and correct a slight error in \cite{EH} Theorem 2.1.6.

Throughout this section, let $\cC$ be a category that has finite limits and let $N,M$ be classes of morphisms in $\cC$.

\subsection{A lifting Lemma}\label{ss_lift}

This subsection is devoted to proving the following lemma:

\begin{lem}\label{l:SpMo_is_Mo}
${}^{\perp}Sp^{\cong}(M) = {}^{\perp}M$.
\end{lem}

\begin{rem}
The idea of the proof of Lemma ~\ref{l:SpMo_is_Mo} appears in \cite{Isa} (see the proof of Lemma 4.11).
\end{rem}

\begin{proof}
Since $M\subseteq Sp^{\cong}(M)$, it is clear that ${}^{\perp}Sp^{\cong}(M)\subseteq {}^{\perp}M$. It remains to show that ${}^{\perp}Sp^{\cong}(M)\supseteq {}^{\perp}M$.
Let $g \in {}^{\perp}M$ and $f \in Sp^{\cong}(M)$. We need to show that $g \perp f$. Without loss of generality we may assume that $f$ comes from a natural transformation $X\to Y$ with the following properties:
    \begin{enumerate}
    \item The indexing category is a cofinite directed set: $T$.
    \item The natural map $X_t \to Y_t \times_{\lim \limits_{s<t} Y_s} \lim \limits_{s<t} X_s $ is in $M$ for every $t\in T$.
    \end{enumerate}
We need to construct a lift in the following diagram:
$$
\xymatrix{
A \ar[d]^g \ar[r] & \{X_t\} \ar[d]^f \\
B          \ar[r] & \{Y_t\}.
}$$
Giving a morphism $B\to \{X_t\}$ means giving morphisms $B\to X_t$ for every $t\in T$, compatible relative to morphisms in $T$, where $X_t$ is regarded as a simple object in $\Pro(\mcal{C})$. Thus, it is enough to construct compatible lifts $B\to X_t$, in the diagrams:
$$
\xymatrix{
A \ar[d]^g \ar[r] & X_t \ar[d]^{f_t} \\
B          \ar[r] & Y_t
}$$
for every $t\in T$.

We will do this by induction on $t$. If $t$ is an element of $T$ such that $d(t)= 0$ (i.t. $t$ is a minimal element of $T$), then such a lift exists since $g \in {}^{\perp}M$, and $$X_t \to  Y_t \times_{\lim \limits_{s<t}  Y_s} \lim \limits_{s<t}  X_s = Y_t$$ is in $M$. Suppose that we have constructed compatible lifts $B\to X_s$, for every $s<t$. Let us construct a compatible lift $B \to X_t$.

We will do this in two stages. First, the compatible lifts $B\to X_s$, for $s<t$, available by the induction hypothesis, gather together to form a lift:
$$
\xymatrix{
A \ar[d]^g \ar[r] & \lim \limits_{s<t} X_s \ar[d]^{f} \\
B  \ar[ru] \ar[r] & \lim \limits_{s<t} Y_s
}$$
and the diagram
$$
\xymatrix{
B\ar[d]  \ar[r]& Y_t \ar[d]\\
\lim \limits_{s<t} X_s \ar[r]   & \lim \limits_{s<t} Y_s
}$$
obviously commutes (since the morphisms $B\to Y_t$ are compatible). Thus we get a lift
$$
\xymatrix{
A\ar[d]^g  \ar[r]&        Y_t \times_{\lim \limits_{s<t} Y_s} \lim \limits_{s<t} X_s \ar[d]\\
B \ar[r] \ar[ru]    & Y_t.
}$$
The second stage is to choose any lift in the square:
$$
\xymatrix{
A\ar[d]^g\ar[r]  \ar[r] & X_t \ar[d]  \\
B \ar[r] & Y_t \times_{\lim \limits_{s<t} Y_s} \lim \limits_{s<t} X_s
}$$
which exists since $g\in {}^\perp M$, and $X_t \to Y_t \times_{\lim \limits_{s<t} Y_s} \lim \limits_{s<t} X_s$ is in $M$.
In particular we get that the following diagram commutes:
$$
\xymatrix{
B \ar[dr] \ar[r]& X_t \ar[d]\\
 & \lim \limits_{s<t} X_s,
}$$
which shows that the lift $B\to X_t$ is compatible.
\end{proof}

\subsection{Reedy type factorizations}\label{ss_reedy}

We now assume that  $M\circ N=Mor(\cC)$. Let $A$ be a cofinite poset and let $f:C\to D$ be a morphism in $\cC^A$. The purpose of this subsection is to describe a construction that produces a factorization of $f$ in $\cC^A$ of the form: $C\xrightarrow{g} H \xrightarrow{h} D$ s.t. $h$ is in $Sp(M)$ and $g$ is in $Lw(N)$ (see Definition ~\ref{def natural}). We will call it the \emph{Reedy construction}. In particular it will follow that  $Sp(M)\circ Lw(N)=Mor(\cC^A)$.

In constructing this factorization we will use the following:
\begin{lem}\label{l:extend_factorization}
Let $R$ be a finite poset, and let $f:X\to Y$ be a map in $\mcal{C}^{R^{\lhd}}$. Let $X|_R \xrightarrow{g} H \xrightarrow{h} Y|_R$ be a factorization of $f|_R$, such that $g$ is levelwise $N$ and $h$ is special $M$.
Then all the factorizations of   $f$  of the form $X \xrightarrow{g'} H' \xrightarrow{h'} Y$, such that $g'$ is levelwise $N$, $h'$ is special $M$ and $H'|_R=H,g'|_R = g,h'|_R = h$, are in natural 1-1 correspondence with all factorizations of the map   $X(\infty) \to \lim\limits_R H \times_{\lim\limits_R Y} Y(\infty)$ of the form $X(\infty) \xrightarrow{g''} H'(\infty) \xrightarrow{h''} \lim\limits_R H \times_{\lim\limits_R Y} Y(\infty)$, s.t. $g''\in N$ and $h'' \in M$ (in particular there always exists one, since $M\circ N=Mor(\cC)$).
\end{lem}

\begin{proof}
To define a factorizations of   $f$  of the form $X \xrightarrow{g'} H' \xrightarrow{h'} Y$ as above, we need to define:
\begin{enumerate}
\item An object: $H'(\infty)\in \cC$.
\item Compatible morphisms: $H'(\infty)\to H(r)$, for every $r\in R$ (or in other words, a morphism: $H'(\infty)\to\lim\limits_R H$).
\item A factorization $X(\infty) \xrightarrow{g'_{\infty}} H'(\infty) \xrightarrow{h'_{\infty}} Y(\infty)$ of $f_{\infty}:X(\infty)\to Y(\infty)$, s.t:
    \begin{enumerate}
    \item The resulting $g':X\to H',h':H'\to Y$ are natural transformations (we only need to check that the following diagram commutes:
        $$\xymatrix{
        X(\infty) \ar[d] \ar[r] &  H'(\infty) \ar[d] \ar[r] & Y(\infty) \ar[d] \\
        \lim\limits_R X  \ar[r] & \lim\limits_R H \ar[r] & \lim\limits_R Y).
        }$$
    \item $g':X\to H'$ is levelwise $N$ (we only need to check that $g'_{\infty}\in N$).
    \item $h':H'\to Y$ is special $M$ (we only need to check the special condition on $\infty\in{R^{\lhd}}$).
    \end{enumerate}
\end{enumerate}
From this the lemma follows easily.
\end{proof}

We define the factorization of $f$ recursively.

Let $n\geq 0$. Suppose we have defined a factorization  of $f|_{A^{n-1}}$ in $\cC^{A^{n-1}}$ of the form: $C|_{A^{n-1}}\xrightarrow{g|_{A^{n-1}}} H|_{A^{n-1}} \xrightarrow{h|_{A^{n-1}}} D|_{A^{n-1}}$, where $h|_{A^{n-1}}$ is in $Sp(M)$ and $g|_{A^{n-1}}$ is in $Lw(N)$ (see Definition ~\ref{def_deg}).

Let $c\in A^n\setminus A^{n-1}$.

$A^{n-1}_c$ is a finite poset, and $f|_{A_c}:C|_{A_c}\to D|_{A_c}$ is a map in $\mcal{C}^{A_c}$ (see Definition ~\ref{def CDS}). $C|_{A^{n-1}_c}\xrightarrow{g|_{A^{n-1}_c}} H|_{A^{n-1}_c} \xrightarrow{h|_{A^{n-1}_c}} D|_{A^{n-1}_c}$ is a factorization of $f|_{A^{n-1}_c}$, such that $g|_{A^{n-1}_c}$ is levelwise $N$ and $h|_{A^{n-1}_c}$ is special $M$.

Note that $A_c$ is naturally isomorphic to $(A^{n-1}_c)^{\lhd}$.
Thus, by Lemma ~\ref{l:extend_factorization}, every factorization of the map   $C(c) \to \lim\limits_{A^{n-1}_c} H \times_{\lim\limits_{A^{n-1}_c} D} D(c)$
into a map in $N$ followed by a map in $M$ gives rise naturally to a factorization of   $f|_{A_c}$  of the form $C|_{A_c}\xrightarrow{g|_{A_c}} H|_{A_c} \xrightarrow{h|_{A_c}} D_{A_c}$ s.t. $g|_{A_c}$ is levelwise $N$ and $h|_{A_c}$ is special $M$, extending the recursively given factorization. Choose such a factorization of  $C(c) \to \lim\limits_{A^{n-1}_c} H \times_{\lim\limits_{A^{n-1}_c} D} D(c)$, and combine all the resulting factorizations  of   $f|_{A_c}$  for different $c\in A^n\setminus A^{n-1}$ to obtain the recursive step.

\subsection{Factorizations in pro categories}\label{ss_fact}

The purpose of this subsection is to prove the rest of Proposition ~\ref{p:fact} not proven in Lemma ~\ref{l:SpMo_is_Mo}. We also explain and correct a slight error in \cite{EH} Theorem 2.1.6.

\begin{prop}\label{p_f}
If $Mor(\mcal{C}) = M\circ N$ then $Mor(Pro(\mcal{C})) = Sp^{\cong}(M)\circ Lw^{\cong}(N)$.
\end{prop}

\begin{proof}
Let $f:X\to Y$ be a morphism in $Pro(\mcal{C})$. By Proposition ~\ref{every map natural} below there exist a natural transformation $f':X'\to Y'$ that is isomorphic to $f$ as a morphism in $Pro(\mcal{C})$. Let $I$ be the mutual indexing category of $X'$ and $Y'$. By Proposition ~\ref{p_onto} below there exist a cofinite directed set $A$ and a cofinal functor: $p:A\to I$. Then $p^*f:p^*X'\to p^*Y'$ is a natural transformation, between diagrams $A\to\cC$, that is isomorphic to $f'$ as a morphism in $Pro(\mcal{C})$. Applying the Reedy construction of Section ~\ref{ss_reedy} to $p^*f$, and composing with the above isomorphisms, we obtain a factorization of $f$ in $Pro(\cC)$ into a morphism in $Lw^{\cong}(N)$ followed by a morphism in $Sp^{\cong}(M)$.
\end{proof}

The proof of Proposition ~\ref{p_f} makes use of the following two classical Propositions:

\begin{prop}\label{every map natural}
Let $f$ be a morphism in $\Pro(\mcal{C})$. Then $f$ is isomorphic, in the category of morphisms in $\Pro(\mcal{C})$, to a morphism that comes from a natural transformation.
\end{prop}
\begin{proof}
See \cite{AM} Appendix 3.2.
\end{proof}

\begin{prop}\label{p_onto}
Let $I$ be a small directed category. Then there exist a (small) cofinite directed set $A$ of infinite hight and a cofinal functor: $p:A\to I$.
\end{prop}

Proposition ~\ref{p_onto} is a well known result in the theory of pro categories. In \cite{Isa}, Isaksen gives two references to this proposition. one is \cite{EH} Theorem 2.1.6 and the other is \cite{SGA4-I} Proposition 8.1.6.

We would like to take this opportunity to explain and correct a slight error in the construction of \cite{EH}. We briefly recall the construction of \cite{EH} Theorem 2.1.6.

Let $\cD$ be any category. Call an object $d\in \cD$ strongly initial, if it is an initial object, and there are no maps into $d$ except the identity. Define:
$$M(I):=\{\cD\to  I|\cD\: is\: finite,\: and\: has\: a\: strongly\: initial\: object\}.$$

We order the set $M(I)$ by sub-diagram inclusion. $M(I)$ is clearly cofinite.

Then \cite{EH} claim that because $I$ is directed, $M(I)$ is also directed. Apparently the idea is that given two diagrams: $F_1:\cD_1\to I,F_2:\cD_2\to I$, we can take the disjoint union of $\cD_1,\cD_2$, and add an initial object: $(\cD_1\coprod\cD_2)^{\lhd}$. In order to define a diagram $(\cD_1\coprod\cD_2)^{\lhd}\to I$ extending $F_1,F_2$, it is thus enough to find an object $F(\infty)\in I$, and morphisms in $I$: $F(\infty)\to F_1(\infty_1),F(\infty)\to F_2(\infty_2)$. Since $I$ is directed this can be done. Notice however, that we have only used the fact that $I$ satisfies one of the axioms of a directed category, namely, that for every pair of objects there is an object that dominates both. If this construction was correct it would mean that for every category $I$ satisfying only the first axiom of a directed category, there exist a directed poset $\cP$ and a cofinal functor $\cP\to I$. This would imply that $I$ is a directed category, by the lemma below.
But there are examples of categories satisfying only the first axiom of a directed category, that are not directed, e.g.
the category $\bullet \rightrightarrows \bullet$ or the category of hyper covers on a Grothendieck site (see ~\cite{AM}).

The reason why this construction is wrong is that $\cD_1,\cD_2$ may not be disjoint (they may have an object in common), and thus one cannot always consider their disjoint union: $\cD_1\coprod\cD_2$. This may sound like a purely technical problem, since we can "force" $\cD_1,\cD_2$ to be disjoint, for example by
considering $(\cD_1\times\{0\})\coprod(\cD_2\times\{1\})$. But then $F_1,F_2$ will not be sub diagrams of $F$, rather there would exist isomorphisms from them to sub diagrams of $F$. In other words, $M(I)$ will not be a poset.

\begin{lem}
Let $A$ be a directed category, $D$ any category and $F:A\to D$ a cofinal functor. Then $D$ is directed.
\end{lem}

\begin{proof}
By \cite{SGA4-I} Proposition 8.1.6, we may assume that $A$ is a directed poset. By \cite{Hir} section 14.2, for every $c\in D$, the over category $F_{/c}$ is nonempty and connected.

Let $c,d\in D$. $F_{/c},F_{/d}$ are non empty, so there exist $q,p\in A$, and morphisms in $D$ of the form:
$$F(q)\to d,F(p)\to c.$$
$A$ is directed, so there exist $r\in A$ s.t. $r\geq p,q$. Then $F(r)\in D$, and we have morphisms in $D$ of the form:
$$F(r)\to F(q)\to d,F(r)\to F(p)\to c.$$

Let $f,g:c\to d$ be two parallel morphisms in $D$. $F_{/c}$ is nonempty, so there exist $p\in A$, and a morphism in $D$ of the form: $h:F(p)\to c$. Then $gh,fh\in F_{/d}$, and $F_{/d}$ is connected, so there exist elements in $A$ of the form:
$$p\leq p_1\geq p_2\leq ... p_n\geq p,$$
that connect $gh,fh:F(p)\to d$ in the over category $F_{/d}$. $A$ is directed, so there exist $q\in A$, s.t. $q\geq p,p_1,...,p_n$. It follows that we have a commutative diagram in $D$ of the form:
$$\xymatrix{
F(p) \ar[dr]_{gh}  & F(q) \ar[l]_{l_1} \ar[r]^{l_2} \ar[d] & F(p)  \ar[dl]^{fh}\\
\empty &  d & \empty .}$$
But, $l_1=l_2=l$, since $A$ is a poset. Define: $t:=hl:F(q)\to c$. then:
$$ft=fhl=ghl=gt.$$
\end{proof}

In order to prove Proposition ~\ref{p_onto} we can still use the construction of \cite{SGA4-I} Proposition 8.1.6. However, we would like to offer an alternative construction, more in the spirit of the construction of \cite{EH}. The main idea is to replace the use of diagrams by an inductive procedure.

\emph{Proof of Proposition ~\ref{p_onto}:} We shall define $A$ and $p:A\to I$ recursively.

We start with defining $A^{-1}:=\phi$, and $p^{-1}:A^{-1}=\phi \to I$ in the only possible way.

Now, suppose we have defined an $n$-level cofinite poset $A^n$, and a functor $p^n:A^n \to I$.

We define $B^{n+1}$ to be the set of all tuples  $(R,p:R^{\lhd}\to I)$ such that
$R$ is a finite Reysha in $A^n$ (see Definition ~\ref{d:Reysha}), $p:R^{\lhd}\to I$ is a functor such that $p|_R = p^n|_R$.

As a set, we define: $A^{n+1} := A^n \coprod B^{n+1}$. For $c\in A^{n}$, we set $c < (R,p:R^{\lhd}\to I)$ iff $c \in R$. Thus we have defined an $(n+1)$-level cofinite poset: $A^{n+1}$.
We now define $p^{n+1}:A^{n+1} \to I$ by $p^{n+1}|_{A^n} = p^{n} $ and $p^{n+1}(R,p:R^{\lhd}\to T) = p(\infty)$, where $\infty \in R^{\lhd}$ is the initial object.
Now we define $A = \cup A^n$.

It is clear that by taking the limit on all the $p^n$ we obtain a functor $p:A \to I$.

Note that $A^0=Ob(I)$ and $p^0:A^{0}=Ob(I) \to I$ is the identity on $Ob(I)$.

\begin{lem}\label{l:A_is_directed}
$A$ is directed.
\end{lem}
\begin{proof}
To prove that $A$ is directed we need to  show that for every finite reysha $R \subset A$,
there exist an element $c \in A$, such that $c\geq r$ for every $r\in R$ (see Lemma ~\ref{l:eqiv_directed}).
Indeed let $R \subset A$ be a finite reysha. Since $R$ is finite, there exist some $n \in \NN$ such that
$R \subset A^{n}$. We can take $c$ to be any element in $B^{n+1}$ of the form $(R,p:R^{\lhd}\to T)$. To show that such an element exists, note that since $I$ is directed we can extend the functor $p^n|_R:R\to I$ to a functor $p:R^{\lhd}\to I$ (see Lemma ~\ref{l:eqiv_directed}).
\end{proof}

\begin{lem}\label{l:A_onto}
The functor: $q:A\to I$ is cofinal.
\end{lem}

\begin{proof}
By \cite{Hir} section 14.2 we need to show that for every $i\in I$, the over category $q_{/i}$ is nonempty and connected. Let $i\in I$.

As noted above, $A^0=Ob(I)$ and $p|_{A^{0}}:Ob(I) \to I$ is the identity on $Ob(I)$. Thus $(i,id_i)$ is an object in $q_{/i}$.

Let $f_1:q(a_1)\to i,f_2:q(a_2)\to i$ be two objects in $q_{/i}$. Since $A$ is directed, there exist $c\in A$ s.t. $c\geq a_1, a_2$. Applying $q$ and composing with $f_1,f_2$ we get two parallel morphisms in $I$: $q(c)\to i$. Since $I$ is directed, there exist a morphism: $h:i'\to q(c)$ in $I$ that equalizes these two parallel morphisms.

We now wish to show that there exist $c'\in A$ s.t. $c'\geq c$ and s.t. $q(c')  =i'$ and the induced map: $q(c') \to q(c)$ is exactly $h$.

There exist a unique $n\geq 0$, s.t. $c\in A^n\setminus A^{n-1}=B^n$.
We can write $c$ as $c = (R,p:R^{\lhd} \to I)$, where $R$ is a finite reysha in $A^{n-1}$.

Note that $R_c:=\{a\in A^n|c\geq a\}\subseteq A^n $ is naturally isomorphic to $R^{\lhd}$.

Define: $c' := (R_c,p':R_c^{\lhd}\to I) \in B^{n+1}$, where:
$$p'|_{R_c} = p'|_{R^{\lhd}} = p|_{R^{\lhd}},p'(\infty') = i'.$$
The map $p'(\infty') = i'\to q(c)=p(\infty)=p'(\infty)$, is defined to be $h$ (where $\infty\in R^{\lhd}$, $\infty'\in R_c^{\lhd}$ are the initial objects).

To show that $c' \in B^{n+1}$, it remains to check that $p'|_{R_c}=p^n|_{R_c}$. But this follows from the fact that $p|_{R}=p^{n-1}|_{R}$, and the (recursive) definition of $p_n$.

Now it is clear that: $c'>c$, $q(c') = i'$ and the induced map: $q(c') \to q(c)$ is exactly $h$.

It follows that we have morphisms in $q_{/i}$:
$$\xymatrix{
q(a_1) \ar[dr]_{f_1}  & q(c') \ar[l] \ar[r] \ar[d] & q(a_2).  \ar[dl]^{f_2}\\
\empty &  i & \empty }$$
\end{proof}

$A$ is clearly of infinite hight, so we have concluded the proof.$\square$

We now continue with the main theme of this section. Our aim is to prove that if $(N,M)$ is a weak factorization system in $\cC$, then $(Lw^{\cong}(N),R(Sp^{\cong}(M)))$ is a weak factorization system in $Pro(\cC)$. For this we will need the following:

\begin{lem}\label{l_lift}
Assume $Mor(\cC)=M\circ N$. Then:
\begin{enumerate}
\item  $N^{\perp}\subseteq R(M)$.
\item  $^{\perp}M\subseteq R(N)$.
\end{enumerate}
\end{lem}
\begin{proof}
We prove (1) and the proof of (2) is dual.
Let $h:A \to B \in N^{\perp}$. We can factor $h$ as:
$$A \xrightarrow{g\in N} C \xrightarrow{f\in M} B.$$
We get the commutative diagram:
$$\xymatrix{
 A \ar[d]_{g\in N}\ar@{=}[r]& A \ar[d]^{h \in N^{\perp}}\\
 C \ar[r]_{f} \ar@{..>}[ru]^{k} & B}$$
where the existence of $k$ is clear. Rearranging, we get:
$$\xymatrix{
 A \ar[dr]^h \ar[r]^g  & C \ar[d]^{f } \ar[r]^k  & A \ar[dl]^h\\
 \empty &  B & \empty,}$$
and we see that $h$ is a retract of $f\in M$.
\end{proof}

\begin{lem}\label{l_lift_wfs}
Assume $Mor(\cC)=M\circ N$ and $N \perp M$. Then $(R(N),R(M))$ is a weak factorization system in $\cC$
\end{lem}
\begin{proof}
$^{\perp}M$ and $N^{\perp}$ are clearly closed under retracts, so by Lemma ~\ref{l_lift} we get that: $R(N) \subseteq ^{\perp}M \subseteq R(N)$ and $R(M)\subseteq N^{\perp} \subseteq R(M)$. Now the Lemma follows from Lemma ~\ref{c:ret_lift}.
\end{proof}

\begin{lem}\label{l:MN_LWSP}
Assume $N \perp M$. Then $Lw^{\cong}(N) \perp Sp^{\cong}(M) $.
\end{lem}
\begin{proof}
Let $f:X\to Y$ be a map in $Lw^{\cong}(N)$. We want to show that $f\in {}^\perp Sp^{\cong}(M)$. But ${}^\perp Sp^{\cong}(M) =  {}^\perp M$ by Lemma ~\ref{l:SpMo_is_Mo}, so it is enough to show that there exist a lift in every square in $Pro(\cC)$ of the form:
$$\xymatrix{
 X \ar[d]_{f}\ar[r]& A \ar[d]^{M}\\
 Y \ar[r]     & B.}$$
Without loss of generality, we may assume that $f:X\to Y$ is a natural transformation, which is is a level-wise $N$-map. Thus we have a diagram of the form:
$$\xymatrix{
\{X_t\}_{t\in T} \ar[d]_{f}\ar[r]& A \ar[d]^{{M}}\\
 \{Y_t\}_{t\in T} \ar[r]    &  B.}$$
By the definition of morphisms in $\Pro(\cC)$, there exist $t \in T$ such that the above square factors as:
$$
\xymatrix{
\{X_t\}_{t\in T} \ar[d]_{f}\ar[r]& X_t \ar[d]^{N}\ar[r]& A \ar[d]^{M}\\
\{Y_t\}_{t\in T}      \ar[r]     & Y_t \ar[r]     &  B.}$$
Since $N\perp M$ we have a lift in the right square of the above diagram, and so a lift in the original square as desired.
\end{proof}

\begin{prop}\label{p_wfs}
If $Mor(\mcal{C}) = M\circ N$ and $N  \perp M$, then $(Lw^{\cong}(N),R(Sp^{\cong}(M)))$ is a weak factorization system in $Pro(\cC)$.
\end{prop}
\begin{proof}
$Mor(\mcal{C}) = M\circ N$ so by Proposition ~\ref{p_f} we have: $Mor(Pro(\mcal{C})) = Sp^{\cong}(M)\circ Lw^{\cong}(N)$.
Thus by Lemmas ~\ref{l_lift_wfs} and ~\ref{l:MN_LWSP} we have: $(R(Lw^{\cong}(N)),R(Sp^{\cong}(M)))$ is a weak factorization system in $Pro(\cC)$.
But by Lemma ~\ref{l:ret_lw} $R(Lw^{\cong}(N)) = Lw^{\cong}(N)$
which completes our proof.
\end{proof}

\section{Another model for $Pro(\cC)$}\label{s:barr}

In this section we will define a category equivalent to $Pro(\cC)$. This category can be thought of as another model for $Pro(\cC)$. This model will be more convenient for our construction of functorial factorizations, and we believe that it might also be convenient for other applications.

Throughout this section we let $\cC$ be an arbitrary category.

\subsection{Definition of $\barr{Pro}(\cC)$}
The purpose of this subsection is to define the category $\barr{Pro}(\cC)$.

\begin{define}
Let $A$ be a cofinite directed set. We will say that $A$ has infinite hight if for every $a\in A$ there exist $a'\in A$ s.t. $a<a'$.
\end{define}

We now wish to define a category which we denote $\barr{Pro}(\cC)$. An object in $\barr{Pro}(\cC)$ is a diagram $F:A\to \cC$, s.t. $A$ is a cofinite directed set of infinite hight. If we say that $F^A$ is an object in $\barr{Pro}(\cC)$ we will mean that $F$ is an object of $\barr{Pro}(\cC)$ and $A$ is its domain.

Let $F^A,G^B$ be objects in $\barr{Pro}(\cC)$. A \emph{pre morphism} from $F$ to $G$ is a pair $(\alpha,\phi)$, s.t. $\alpha:B\to A$ is a strictly increasing function, and $\phi:\alpha^*F=F\circ\alpha\to G$ is a natural transformation.

\begin{rem}\label{r_strictly}
The reason for demanding a \emph{strictly} increasing function in the definition of a pre morphism will not be clear until much later. See for example the construction of the functor: $\barr{Pro}(\cC^{\Delta^1})\to \barr{Pro}(\cC^{\Delta^2})$ in Section ~\ref{s_ff}.
\end{rem}

We now define a partial order on the set of pre morphisms from $F$ to $G$. Let $(\alpha,\phi)$,$(\alpha',\phi')$ be pre morphisms from $F$ to $G$. Then $(\alpha',\phi')\geq(\alpha,\phi)$ if for every $b\in B$ we have $\alpha'(b)\geq\alpha(b)$, and the following diagram commutes:
\[
\xymatrix{ & F(\alpha(b)) \ar[dr]^{\phi_b} & \\
F(\alpha'(b))\ar[ur]\ar[rr]^{\phi'_b} & & G(b),}
\]
(the arrow $F(\alpha'(b))\to F(\alpha(b))$ is of course the one induced by the unique morphism $\alpha'(b)\to \alpha(b)$ in $A$).

It is not hard to check that we have turned the set of pre morphisms from $F$ to $G$ into a poset. We define a \emph{morphism} from $F$ to $G$ in $\barr{Pro}(\cC)$ to be a connected component of this poset. If $(\alpha,\phi)$ is a pre morphisms from $F$ to $G$, we denote its connected component by $[\alpha,\phi]$. Let $[\alpha,\phi]:F^A\to G^B$ and $[\beta,\psi]:G^B\to H^C$ be morphisms in $\barr{Pro}(\cC)$. Their composition is defined to be $[\alpha\circ\beta,\psi\circ\phi_{\beta}]$. It is not hard to check that this is well defined, and turns $\barr{Pro}(\cC)$ into a category (note that if $(\alpha',\phi')\geq(\alpha,\phi)$ then $(\alpha'\circ\beta,\psi\circ\phi'_{\beta})\geq(\alpha\circ\beta,\psi\circ\phi_{\beta})$, and if $(\beta',\psi')\geq(\beta,\psi)$ then $(\alpha\circ\beta',\psi'\circ\phi_{\beta'})\geq(\alpha\circ\beta,\psi\circ\phi_{\beta})$).

\subsection{Equivalence of $\barr{Pro}(\cC)$ and $Pro(\cC)$}

In this subsection we construct a natural functor $i:\barr{Pro}(\cC)\to Pro(\cC)$. We then show that $i$ is a subcategory inclusion, that is essentially surjective. It follows that $i$ is a categorical equivalence.

Let $F:A\to \cC$ be an object in $\barr{Pro}(\cC)$. Then clearly $i(F):=F$ is also an object ${Pro}(\cC)$.

Let $F^A,G^B$ be objects in $\barr{Pro}(\cC)$, and let $(\alpha,\phi)$ be a pre morphism from $F$ to $G$. Then $(\alpha,\phi)$ determines a morphism $F\to G$ in $\Pro(C)$ (for every $b\in B$ take the morphism $\phi_b:F_{\alpha(b)}\to G_b$). Suppose now that $(\alpha',\phi')$ is another pre morphism from $F$ to $G$, s.t. $(\alpha',\phi')\geq(\alpha,\phi)$. Then it is clear from the definition of the partial order on pre morphisms, that for every $b\in B$ the morphisms $\phi_b:F(\alpha(b))\to G(b)$ and $\phi'_b:F(\alpha'(b))\to G(b)$ represent the same object in $\colim_{i\in A}\Hom_{\cC}(F(i),G(b))$. Thus $(\alpha',\phi')$ and $(\alpha,\phi)$ determine the same morphism $F\to G$ in $Pro(\cC)$. It follows, that a morphism $F\to G$ in $\barr{Pro}(\cC)$ determines a well defined morphism $i(F)\to i(G)$ in $Pro(\cC)$ through the above construction. This construction clearly commutes with compositions and identities, so we have defined a functor: $i:\barr{Pro}(\cC)\to Pro(\cC)$.

\begin{prop}\label{p_full}
The functor $i:\barr{Pro}(\cC)\to Pro(\cC)$ is full.
\end{prop}
\begin{proof}
Let $F^A,G^B$ be objects in $\barr{Pro}(\cC)$. Let $f:F\to G$ be a morphism in $Pro(\cC)$. We need to construct a pre morphism $(\alpha,\phi)$ from $F$ to $G$ that induces our given $f$.

We will define $\alpha:B\to A$, and $\phi:F\circ\alpha\to G$ recursively.

Let $n\geq 0$. Suppose we have defined a strictly increasing function $\alpha:B^{n-1}\to A$, and a natural transformation $\phi:F\circ\alpha\to G|_{B^{n-1}}$, s.t. for every $b\in B^{n-1}$ the morphism $\phi_b:F(\alpha(b))\to G(b)$ represents $f$ (see Definition ~\ref{def_deg} and the remarks after Definition ~\ref{def_pro}).

Let $b\in B^n \setminus B^{n-1}$. We prove the following:
\begin{lem}
There exist $\alpha(b)\in A$ and a morphism $\phi_b:F(\alpha(b))\to G(b)$ representing $f$, s.t. for every $b'\in B_b^{n-1}$ (see Definition ~\ref{def CDS}) we have $\alpha(b)> \alpha(b')$ and the following diagram commutes:
$$\xymatrix{F(\alpha(b))\ar[r]^{\phi_{b}}\ar[d]_{F(\alpha(b)\to \alpha(b'))} & G(b)\ar[d]^{G(b\to b_{i+1})} \\
F(\alpha(b'))\ar[r]^{\phi_{b'}} & G(b').}$$
\end{lem}
\begin{proof}
Write $B_b^{n-1}=\{b_1,...,b_k\}$. We will prove the following by induction on $i$:
For every $i=0,...,k$ there exist $a_i\in A$ and a morphism $F(a_i)\to G(b)$ representing $f$, s.t. for every $1\leq j\leq i$ we have $a_i\geq \alpha(b_j)$ and the following diagram commutes:
$$\xymatrix{F(a_i)\ar[r]\ar[d]_{F(a_{i}\to \alpha(b_{j}))} & G(b)\ar[d]^{G(b\to b_{j})} \\
F(\alpha(b_j))\ar[r]^{\phi_{b_j}} & G(b_j)}.$$

$i=0$. Choose $a_0\in A$ and a morphism $F(a_0)\to G(b)$ representing $f$.

Suppose we have proved the above for some $i\in\{0,...,k-1\}$.

We will prove the above for $i+1$. The morphisms $F(a_i)\to G(b)$ and $\phi_{b_{i+1}}:F(\alpha(b_{i+1}))\to G(b_{i+1})$ both represent $f$. $b\geq b_{i+1}$, so the compatibility of the representing morphisms implies that $\phi_{b_{i+1}}$ and the composition
$$F(a_i)\to G(b)\xrightarrow{G(b\to b_{i+1})}G(b_{i+1})$$
represent the same element in $\colim_{a\in A}\Hom_{\cC}(F(a),G(b_{i+1}))$. Thus, there exist $a_{i+1}\in A$ s.t. $a_{i+1}\geq a_i, \alpha(b_{i+1})$ and the following diagram commutes:
$$\xymatrix{                                                                & F(a_i)\ar[dr]            &                               & \\
F(a_{i+1})\ar[ur]^{F(a_{i+1}\to a_i)}\ar[dr]_{F(a_{i+1}\to \alpha(b_{i+1}))}&                          & G(b)\ar[dr]^{G(b\to b_{i+1})} & \\
                                                                            & F(\alpha(b_{i+1}))\ar[rr]^{\phi_{b_{i+1}}} &             & G(b_{i+1})}.$$
It is not hard to verify that taking $F(a_{i+1})\to G(b)$ to be the morphism described in the diagram above finishes the inductive step.

Since $A$ has infinite hight we can find $\alpha(b)\in A$ s.t. $\alpha(b)>a_k$. Defining $\phi_b$ to be the composition:
$$F(\alpha(b))\xrightarrow{F(\alpha(b)\to a_{k})}F(a_{k})\to G(b),$$
finishes the proof of the lemma.
\end{proof}
The above lemma completes the recursive definition, and thus the proof of the proposition.
\end{proof}

We now wish to prove that $i$ is faithful. We will prove a stronger result:
\begin{prop}\label{p_directed}
Let $F^A,G^B$ be objects in $\barr{Pro}(\cC)$, and let $(\alpha,\phi)$,$(\alpha',\phi')$ be pre morphisms from $F$ to $G$. Assume that $(\alpha,\phi)$ and $(\alpha',\phi')$ induce the same morphism $f:F\to G$ in $Pro(\cC)$. Then there exist a  pre morphism $(\alpha'',\phi'')$ from $F$ to $G$ s.t. $(\alpha'',\phi'')\geq(\alpha,\phi),(\alpha',\phi')$.
\end{prop}
\begin{proof}
We will define $\alpha'':B\to A$ and $\phi'':F\circ\alpha\to G$ recursively.

Let $n\geq 0$. Suppose we have defined a strictly increasing function $\alpha'':B^{n-1}\to A$ and a natural transformation $\phi'':F\circ\alpha''\to G|_{B^{n-1}}$, s.t. for every $b\in B^{n-1}$ we have $\alpha''(b)\geq \alpha(b),\alpha'(b)$ and the following diagram commutes:
$$\xymatrix{ &   F(\alpha(b))\ar[dr]^{\phi_b}     & \\
F(\alpha''(b))\ar[ur]^{F(\alpha''(b)\to \alpha(b))}\ar[dr]_{F(\alpha''(b)\to \alpha'(b))} \ar[rr]^{\phi''_b} & & G(b) \\
             &   F(\alpha'(b))\ar[ur]_{\phi'_b}   & }$$
(see Definition ~\ref{def_deg}).

Let $b\in B^n \setminus B^{n-1}$.
We prove the following:
\begin{lem}
There exist $\alpha''(b)\in A$ and a morphism $\phi''_b:F(\alpha''(b))\to G(b)$, s.t. for every $b'\in B^{n-1}_b$ (see Definition ~\ref{def CDS}) we have $\alpha''(b)> \alpha''(b')$ and the following diagram commutes:
$$\xymatrix{F(\alpha''(b))\ar[r]^{\phi''_b}\ar[d]_{F(\alpha''(b)\to \alpha''(b'))} & G(b)\ar[d]^{G(b\to b')} \\
F(\alpha''(b'))\ar[r]^{\phi''_{b'}} & G(b'),}$$
and we have $\alpha''(b)\geq \alpha(b), \alpha'(b)$ and the following diagram commutes:
$$\xymatrix{ &   F(\alpha(b))\ar[dr]^{\phi_b}     & \\
F(\alpha''(b))\ar[ur]^{F(\alpha''(b)\to \alpha(b))}\ar[dr]_{F(\alpha''(b)\to \alpha'(b))}\ar[rr]^{\phi''_b} & & G(b). \\
             &   F(\alpha'(b))\ar[ur]_{\phi'_b}   & }$$
\end{lem}
\begin{proof}
Write $B_b^{n-1}=\{b_1,...,b_k\}$. We will prove the following by induction on $i$:
For every $i=0,...,k$ there exist $a_i\in A$ and a morphism $F(a_i)\to G(b)$, s.t. for every $1\leq j\leq i$ we have $a_i\geq \alpha''(b_j)$ and the following diagram commutes:
$$\xymatrix{F(a_i)\ar[r]\ar[d]_{F(a_{i}\to \alpha''(b_{j}))} & G(b)\ar[d]^{G(b\to b_{j})} \\
F(\alpha''(b_j))\ar[r]^{\phi''_{b_j}} & G(b_j)},$$
and we have $a_i \geq \alpha(b), \alpha'(b)$ and the following diagram commutes:
$$\xymatrix{ &   F(\alpha(b))\ar[dr]^{\phi_b}     & \\
F(a_i)\ar[ur]^{F(a_i\to \alpha(b))}\ar[dr]_{F(a_i\to \alpha'(b))}\ar[rr] & & G(b). \\
             &   F(\alpha'(b))\ar[ur]_{\phi'_b}   & }$$

$i=0$. The morphisms $\phi_b:F(\alpha(b))\to G(b)$ and $\phi'_b:F(\alpha'(b))\to G(b)$ represent the same element in $\colim_{a\in A}\Hom_{\cC}(F(a),G(b))$. It follows that there exist $a_0\in A$ s.t. $a_0\geq \alpha(b), \alpha'(b)$ and the following diagram commutes:
$$\xymatrix{ &   F(\alpha(b))\ar[dr]^{\phi_b}     & \\
F(a_0)\ar[ur]^{F(a_0\to \alpha(b))}\ar[dr]_{F(a_0\to \alpha'(b))} & & G(b). \\
             &   F(\alpha'(b))\ar[ur]_{\phi'_b}   & }$$
We thus \emph{define} the morphism $F(a_0)\to G(b)$ to be the one described in the diagram above.

Suppose we have proved the above for some $i\in\{0,...,k-1\}$.

We will prove the above for $i+1$. The morphisms $F(a_i)\to G(b)$ and $\phi''_{b_{i+1}}:F(\alpha''(b_{i+1}))\to G(b_{i+1})$ both represent $f$. $b\geq b_{i+1}$, so the compatibility of the representing morphisms implies that $\phi''_{b_{i+1}}$ and the composition
$$F(a_i)\to G(b)\xrightarrow{G(b\to b_{i+1})}G(b_{i+1})$$
represent the same object in $\colim_{a\in A}\Hom_{\cC}(F(a),G(b_{i+1}))$. Thus, there exist $a_{i+1}\in A$ s.t. $a_{i+1}\geq a_i, \alpha''(b_{i+1})$ and the following diagram commutes:
$$\xymatrix{                                                                & F(a_i)\ar[dr]            &                               & \\
F(a_{i+1})\ar[ur]^{F(a_{i+1}\to a_i)}\ar[dr]_{F(a_{i+1}\to \alpha''(b_{i+1}))}&                          & G(b)\ar[dr]^{G(b\to b_{i+1})} & \\
                                                                            & F(\alpha''(b_{i+1}))\ar[rr]^{\phi''_{b_{i+1}}} &             & G(b_{i+1})}.$$
It is not hard to verify that taking $F(a_{i+1})\to G(b)$ to be the morphism described in the diagram above finishes the inductive step.

Since $A$ has infinite hight we can find $\alpha''(b)\in A$ s.t. $\alpha''(b)>a_k$. Defining $\phi''_b$ to be the composition:
$$F(\alpha''(b))\xrightarrow{F(\alpha''(b)\to a_{k})}F(a_{k})\to G(b),$$
finishes the proof of the lemma.
\end{proof}
The above lemma completes the recursive definition, and thus the proof of the proposition.
\end{proof}

\begin{cor}
The functor $i:\barr{Pro}(\cC)\to Pro(\cC)$ is faithful.
\end{cor}

\begin{cor}\label{c_directed}
Let $F,G$ be objects in $\barr{Pro}(\cC)$. Then every connected component of the poset of pre morphisms from $F$ to $G$ (that is, every morphism from $F$ to $G$ in $\barr{Pro}(\cC)$) is directed.
\end{cor}


We have shown that the functor $i:\barr{Pro}(\cC)\to Pro(\cC)$ is (isomorphic to) a full subcategory inclusion. From Proposition ~\ref{p_onto} it follows immediately that $i$ is essentially surjective on objects. Thus we obtain the following:

\begin{cor}\label{c_equiv}
The functor $i:\barr{Pro}(\cC)\to Pro(\cC)$ is (isomorphic to) a full subcategory inclusion, and is essentially surjective on objects. Thus it is an equivalence of categories.
\end{cor}

\section{Functorial factorizations in pro categories}\label{s_ff}
The purpose of this section is to prove Theorem ~\ref{t_main}.  It is done in Theorem ~\ref{t_fact} and Corollary ~\ref{c_wfs}.

Throughout this section, let $\cC$ be a category that has finite limits and let $N,M$ be classes of morphisms in $\cC$. We begin with some definitions:

\begin{define}
For any $n\geq 0$ let $\Delta^n$ denote the linear poset: $\{0,...,n\}$, considered as a category with a unique morphism $i\to j$ for any $i\leq j$.
\end{define}

\begin{define}\label{d_ff}
Let $\cD$ be a category.

A functorial factorization in $\cD$,  is a section to the composition functor: $\circ:\cD^{\Delta^2}\to \cD^{\Delta^1}$ (which is the pull back to the inclusion: ${\Delta^1}\cong\Delta^{\{0,2\}}\hookrightarrow{\Delta^2}$).

Thus a functorial factorization in $\cD$ consists of a functor:
$\cD^{\Delta^1}\to\cD^{\Delta^2}$ denoted:
$$(X\xrightarrow{f}Y)\mapsto ({X}\xrightarrow{q_f} L_{f}\xrightarrow{p_f}Y)$$
s.t.
\begin{enumerate}
\item For any morphism $f$ in $\cD$ we have: $f=p_f\circ q_f$.
\item For any morphism:
$$\xymatrix{{X}\ar[r]^{f}\ar[d]_{l} & {Y}\ar[d]^{k}\\
            {Z} \ar[r]^t   &   {W},}$$
in $\cD^{\Delta^1}$ the corresponding morphism in $\cD^{\Delta^2}$ is of the form:
$$\xymatrix{{X}\ar[r]^{q_f}\ar[d]_{l} & L_{f}\ar[r]^{h_f}\ar[d]^{L_{(l,k)}} & {Y}\ar[d]^{k}\\
            {Z} \ar[r]^{q_t} & L_t\ar[r]^{p_t}    &   {W}.}$$
\end{enumerate}

Suppose $\cA$ and $\cB$ are classes of morphisms in $\cD$. The above functorial factorization is said to be into a morphism in $\cA$ followed be a morphism in $\cB$, if for every $f\in Mor(\cC)$ we have $q_f\in \cA,p_f\in \cB$.

\end{define}

Notice, that a functorial factorization in $\cD$ is just a morphism in the over category $Cat_{/\cD^{\Delta^1}}$ $id|_{\cD^{\Delta^1}}\to\circ$, where $id|_{\cD^{\Delta^1}}\in Cat_{/\cD^{\Delta^1}}$ is the terminal object and $\circ:\cD^{\Delta^2}\to \cD^{\Delta^1}$ is the composition functor.

\begin{rem}
The definition above of a functorial factorization agrees with the one given in \cite{Rie}. It is slightly stronger then the one given in \cite{Hov} Definition 1.1.1.
\end{rem}

For technical reasons we will also consider the following weaker notion:
\begin{define}\label{d_wff}
Let $\cD$ be a category.

A weak functorial factorization in $\cD$ is a section, up to a natural isomorphism, to the composition functor: $\circ:\cD^{\Delta^2}\to \cD^{\Delta^1}$.

If $\cA$ and $\cB$ are classes of morphisms in $\cD$, we can define a weak functorial factorization into a morphism in $\cA$ followed be a morphism in $\cB$, in the same way as in Definition ~\ref{d_ff}.
\end{define}

Consider the category $Cat$ as an $(2,1)$-category where the morphisms are functors, and the $2$-morphisms are natural isomorphisms. Then the homotopy category $hCat$ is obtained from the usual category $Cat$ by identifying functors that are naturally isomorphic.

Then a weak functorial factorization in $\cD$ is just a morphism in the category $hCat_{/\cD^{\Delta^1}}$: $id|_{\cD^{\Delta^1}}\to\circ$ (we abuse notation and identify a weak functorial factorization: $\cD^{\Delta^1}\to\cD^{\Delta^2}$ with its image in $hCat$).

\begin{lem}
To any weak functorial factorization in $\cD$ there exist a functorial factorization in $\cD$ isomorphic to it.
\end{lem}
\begin{proof}
Let:
$$(X\xrightarrow{f}Y)\mapsto (\barr{X}\xrightarrow{q_f} L_{f}\xrightarrow{p_f}\barr{Y}),$$
be a weak functorial factorization in $\cD$.

There is a natural isomorphism between the identity and the composition of the above factorization with the composition functor. Thus, for any morphism $f$ in $\cD$ we have an isomorphism: $i_f : f\xrightarrow{\cong}p_f\circ q_f$ in $\cD^{\Delta^1}$ denoted:
$$\xymatrix{{X}\ar[r]^{(i_f)_0}\ar[d]_{f} & \barr{X}\ar[d]^{p_f\circ q_f}\\
            {Y} \ar[r]^{(i_f)_1}   &   \barr{Y},}$$
s.t. for any morphism:
$$\xymatrix{{X}\ar[r]^{f}\ar[d]_{} & {Y}\ar[d]^{}\\
            {Z} \ar[r]^t   &   {W},}$$
in $\cD^{\Delta^1}$ the following diagram commutes:
$$\xymatrix{  & X\ar[dl]^{(i_f)_0}\ar[dd]\ar[rr]^f   &   &   Y\ar[dl]^{(i_f)_1}\ar[dd]\\
          \barr{X}\ar[dd]\ar[rr]^{p_f\circ q_f} &  & \barr{Y}\ar[dd] &           \\
            & Z\ar[dl]^{(i_t)_0}\ar[rr]^t   &   &   W \ar[dl]^{(i_t)_1}\\
          \barr{Z}\ar[rr]^{p_t\circ q_t} &  & \barr{W} & }$$

We define a functorial factorization in $\cD$ by:
$$(X\xrightarrow{f}Y)\mapsto ({X}\xrightarrow{q_f\circ (i_f)_0} L_{f}\xrightarrow{(i_f)_1^{-1}\circ p_f}{Y}).$$
For any morphism $f$ in $\cD$ we have a commutative diagram:
$$\xymatrix{{X}\ar[r]^{q_f\circ (i_f)_0}\ar[d]^{(i_f)_0}_{\cong} & L_{f}\ar[r]^{(i_f)_1^{-1}\circ p_f}\ar[d]_{=} & {Y}\ar[d]^{(i_f)_1}_{\cong}\\
           \barr{X}\ar[r]^{q_f} & L_{f}\ar[r]^{p_f} & \barr{Y},}$$
so the proof is complete.
\end{proof}

\begin{cor}\label{c_wff_to_ff}
Let $\cA$ and $\cB$ be classes of morphisms in $\cD$, that are invariant under isomorphisms. If there exist a weak functorial factorization in $\cD$ into a morphism in $\cA$ followed be a morphism in $\cB$, then there exist a functorial factorization in $\cD$ into a morphism in $\cA$ followed be a morphism in $\cB$.
\end{cor}

We are now ready to prove our main theorem:

\begin{thm}\label{t_fact}
If $Mor(\mcal{C}) =^{func} M\circ N$ then $Mor(Pro(\mcal{C})) =^{func} Sp^{\cong}(M)\circ Lw^{\cong}(N)$.
\end{thm}

\begin{proof}
Assume that we are given a functorial factorization in $\cC$ into a morphism in $N$ followed by a morphism in $M$. We need to find a functorial factorization in $Pro(\cC)$ into a morphism in $Lw^{\cong}(N)$ followed by a morphism in $Sp^{\cong}(M)$ (see Definition ~\ref{d_ff}).


Since $Lw^{\cong}(N)$ and $Sp^{\cong}(M)$ are clearly invariant under isomorphisms, Corollary ~\ref{c_wff_to_ff} implies that it is enough to find a \emph{weak} functorial factorization in $Pro(\cC)$ into a morphism in $Lw^{\cong}(N)$ followed by a morphism in $Sp^{\cong}(M)$.

Recall that a weak functorial factorization in $Pro(\cC)$ is just a morphism in the category $hCat_{/Pro(\cC)^{\Delta^1}}$: $id|_{Pro(\cC)^{\Delta^1}}\to\circ$. 

We will achieve our goal by first  replacing $Pro(\cC)^{\Delta^1}$ and  $Pro(\cC)^{\Delta^2}$ with equivalent categories.

First, for every small category $A$ there is a a natural functor:

$$p_A:Pro(\cC^{A})\to Pro(\cC)^{A}.$$
By \cite{Mey},
when $A$ is a finite loopless category (for e.g. $A = \Delta^n$) $p_A$ is an equivalence of categories.

Consider now the following commutative diagram:

$$
\xymatrix{
\barr{Pro}(\cC^{\Delta^2}) \ar[r]^{\sim} \ar[d]^{[\circ_1]}& Pro(\cC^{\Delta^2}) \ar[r]^{\sim}_{p_{\Delta^2}} \ar[d]^{[\circ_2]} & Pro(\cC)^{\Delta^2}\ar[d]^{[\circ_3]} \\
\barr{Pro}(\cC^{\Delta^1}) \ar[r]^{\sim} & Pro(\cC^{\Delta^1}) \ar[r]^{\sim}_{p_{\Delta^1}} & Pro(\cC)^{\Delta^1},
}$$
where the $\circ_i$ are the different morphisms induced from composition.

We see now that out goal is to construct  a section $s_3$ to $\circ_3$ up to a natural transformation.
Note that for this it is enough to find a section $s_1$ to $\circ_1$. Indeed assume we have such an $s_1$ and consider the commutative  diagram:
$$
\xymatrix{
\barr{Pro}(\cC^{\Delta^2}) \ar[r]^{\sim}_{e_{\Delta^2}} \ar[d]^{[\circ_1]}& Pro(\cC)^{\Delta^2}\ar[d]^{[\circ_3]} \\
\barr{Pro}(\cC^{\Delta^1}) \ar[r]^{\sim}_{e_{\Delta^1}} & Pro(\cC)^{\Delta^1}.
}$$.

Since $e_{\Delta^1}$ is an equivalence we can  choose some functor $h_{\Delta^1}$ such that:
$$ e_{\Delta^1}\circ h_{\Delta^1} \sim Id_{Pro(\cC)^{\Delta^1}}.$$

Now take $s_3 := e_{\Delta^2}\circ s_1\circ h_{\Delta_1}$ and we get:

$$ [\circ_3]\circ s_3 = [\circ_3]\circ  e_{\Delta^2}\circ s_1\circ h_{\Delta_1} = $$
$$ = e_{\Delta^1} \circ [\circ_1] \circ s_1\circ h_{\Delta_1} = $$
$$= e_{\Delta^1} \circ h_{\Delta_1} \sim Id_{Pro(\cC)^{\Delta^1}}.$$

So we are left with constructing  the section:

%
%

$$s_1:\barr{Pro}(\cC^{\Delta^1})\to \barr{Pro}(\cC^{\Delta^2}).$$ 
 
Let $f$ be an object of $\barr{Pro}(\cC^{\Delta^1})$. Then $f:E^A\to F^A$ is a natural transformation between objects in $\barr{Pro}(\cC)$. We define the value of our functor on $f$ to be the Reedy factorization: $E\xrightarrow{g_f} H_f\xrightarrow{h_f} F$ (see Section ~\ref{ss_reedy}). As we have shown, we have: $f=h_f\circ g_f,g_f\in Lw(N),h_f\in Sp(M)$.

Let $f$ and $t$ be objects of $\barr{Pro}(\cC^{\Delta^1})$. Then $f:E^A\to F^A,t:K^B\to G^B$ are natural transformation between objects in $\barr{Pro}(\cC)$. Let $(\alpha,\Phi)$ be a representative to a morphism $f\to t$ in $\barr{Pro}(\cC^{\Delta^1})$. Then $\alpha:B\to A$ is a strictly increasing function and $\Phi:\alpha^*f\to t$ is a morphism in $(\cC^{\Delta^1})^B\cong(\cC^B)^{\Delta^1}.$ Thus $\Phi=(\phi,\psi)$ is just a pair of morphisms in $\cC^B$ and we have a commutative diagram in $\cC^B$:
$$\xymatrix{{E\circ\alpha}\ar[r]^{{f_{\alpha}}}\ar[d]_{\phi} & {F\circ\alpha}\ar[d]^{\psi}\\
            K \ar[r]^t    &   G.}$$

Now consider the Reedy factorizations of $f$ and $t$:
$$E\xrightarrow{g_f} H_f\xrightarrow{h_f} F,K\xrightarrow{g_t} H_t\xrightarrow{h_t} G.$$
We need to construct a representative to a morphism in $\barr{Pro}(\cC^{\Delta^2})$ between these Reedy factorizations. We take the strictly increasing function $B\to A$ to be just $\alpha$. All we need to construct is a natural transformation: $\chi:H_f\circ\alpha\to H_t$ such that the following diagram in $\cC^B$ commutes:
$$\xymatrix{{E}\circ\alpha\ar[r]^{(g_f)_{\alpha}}\ar[d]_{\phi} & H_{f}\circ\alpha\ar[r]^{(h_f)_{\alpha}}\ar[d]^{\chi} & {F}\circ\alpha\ar[d]^{\psi}\\
            {K} \ar[r]^{g_t} & H_t\ar[r]^{h_t}    &   {G}.}$$
We will define $\chi:H_f\circ\alpha\to H_t$ recursively, and refer to it as the $\chi$-construction.

Let $n\geq 0$. Suppose we have defined a natural transformation: $\chi:(H_f\circ\alpha)|_{B^{n-1}}\to H_t|_{B^{n-1}}$ such that the following diagram in $\cC^{B^{n-1}}$ commutes:
$$\xymatrix{({E}\circ\alpha)|_{B^{n-1}}\ar[r]^{(g_f)_{\alpha}}\ar[d]_{\phi} & (H_{f}\circ\alpha)|_{B^{n-1}}\ar[r]^{(h_f)_{\alpha}}\ar[d]^{\chi} & ({F}\circ\alpha)|_{B^{n-1}}\ar[d]^{\psi}\\
            {K}|_{B^{n-1}} \ar[r]^{g_t} & H_t|_{B^{n-1}}\ar[r]^{h_t}    &   {G}|_{B^{n-1}}}$$
(see Definition ~\ref{def_deg}).

Let $b\in B^n\setminus B^{n-1}$.

There exist a unique $m\geq0$ s.t. $\alpha(b)\in A^m\setminus A^{m-1}$. It is not hard to check, using the induction hypothesis and the assumptions on the datum we began with, that we have an induced commutative diagram:
$$\xymatrix{  & E(\alpha(b))\ar[dl]\ar[dd]\ar[rr]   &   &   K(b)\ar[dl]\ar[dd]\\
          F(\alpha(b))\ar[dd]\ar[rr] &  & G(b)\ar[dd] &           \\
            & \lim_{A_{\alpha(b)}^{m-1}}H_f\ar[dl]^r\ar[rr]   &   &   \lim_{B_{b}^{n-1}}H_t\ar[dl]^s\\
          \lim_{A_{\alpha(b)}^{m-1}}F\ar[rr] &  & \lim_{B_{b}^{n-1}}G}$$
(we remark that one of the reasons for demanding a \emph{strictly} increasing function in the definition of a pre morphism is that otherwise we would not have the two bottom horizontal morphism in the above diagram, see Remark ~\ref{r_strictly})

Thus, there is an induced commutative diagram:
$$\xymatrix{ E(\alpha(b))\ar[d]\ar[r]   &  \lim_{A_{\alpha(b)}^{m-1}}H_f\times_{\lim_{A_{\alpha(b)}^{m-1}}F} F(\alpha(b))\ar[d] \\
 K(b)\ar[r] & \lim_{B_{b}^{n-1}}H_t\times_{\lim_{B_{b}^{n-1}}G} G(b).}$$

We apply the functorial factorizations in $\cC$ to the horizontal arrows in the above diagram and get a commutative diagram:
$$\xymatrix{ E(\alpha(b))\ar[d]\ar[r]   & H_f(\alpha(b))\ar[d]\ar[r]   & \lim_{A_{\alpha(b)}^{m-1}}H_f\times_{\lim_{A_{\alpha(b)}^{m-1}}F} F(\alpha(b))\ar[d] \\
 K(b)\ar[r] & H_t(b)\ar[r] & \lim_{B_{b}^{n-1}}H_t\times_{\lim_{B_{b}^{n-1}}G} G(b).}$$

It is not hard to verify that taking $\chi_b:H_f(\alpha(b))\to H_t(b)$ to be the morphism described in the diagram above completes the recursive definition.

We need to show that the morphism we have constructed in $\barr{Pro}(\cC^{\Delta^2})$ between the Reedy factorizations does not depend on the choice of representative $(\alpha,\Phi)$ to the morphism $f\to t$ in $\barr{Pro}(\cC^{\Delta^1})$. So let $(\alpha',\Phi')$ be another pre morphism from $f$ to $t$.

Thus, $\alpha':B\to A$ is a strictly increasing function, $\Phi'=(\phi',\psi')$ is a pair of morphisms in $\cC^B$ and we have a commutative diagram in $\cC^B$:
$$\xymatrix{{E\circ\alpha'}\ar[r]^{{f_{\alpha'}}}\ar[d]_{\phi'} & {F\circ\alpha'}\ar[d]^{\psi'}\\
            K \ar[r]^t    &   G.}$$

We apply the $\chi$-construction to this new datum and obtain a natural transformation: $\chi':H_f\circ\alpha'\to H_t$.

\begin{lem}\label{l_reedy wd}
Suppose that $(\alpha',\Phi')\geq(\alpha,\Phi)$.
Then for every $b\in B$ we have: $\alpha'(b)\geq\alpha(b)$ and the following diagram commutes:
\[
\xymatrix{ & H_f(\alpha(b)) \ar[dr]^{\chi_b} & \\
H_f(\alpha'(b))\ar[ur]\ar[rr]^{\chi'_b} & & H_t(b).}
\]
In other words, we have an inequality of pre morphisms from $H_f$ to $H_t$:
$$(\alpha',\chi')\geq(\alpha,\chi)$$
\end{lem}
\begin{proof}
$(\alpha',\Phi')\geq(\alpha,\Phi)$ means that for every $b\in B$ we have: $\alpha'(b)\geq\alpha(b)$ and the following diagrams commute:
\[
\xymatrix{ & E(\alpha(b)) \ar[dr]^{\phi_b} &  &  &   F(\alpha(b)) \ar[dr]^{\psi_b} &  \\
E(\alpha'(b))\ar[ur]\ar[rr]^{\phi'_b} & & K(b) &  F(\alpha'(b))\ar[ur]\ar[rr]^{\psi'_b} & & G(b).}
\]

We will prove the conclusion inductively.

Let $n\geq 0$. Suppose we have shown that for every $b\in B^{n-1}$ the following diagram commutes:
\[
\xymatrix{ & H_f(\alpha(b)) \ar[dr]^{\chi_b} & \\
H_f(\alpha'(b))\ar[ur]\ar[rr]^{\chi'_b} & & H_t(b).}
\]

Let $b\in B^n\setminus B^{n-1}$.

There exist a unique $m\geq0$ s.t. $\alpha(b)\in A^m\setminus A^{m-1}$. As we have shown, we have an induced commutative diagram:
$$\xymatrix{ E(\alpha(b))\ar[d]\ar[r]   &  \lim_{A_{\alpha(b)}^{m-1}}H_f\times_{\lim_{A_{\alpha(b)}^{m-1}}F} F(\alpha(b))\ar[d] \\
 K(b)\ar[r] & \lim_{B_{b}^{n-1}}H_t\times_{\lim_{B_{b}^{n-1}}G} G(b),}$$
and the map $\chi_b:H_f(\alpha(b))\to H_t(b)$ is just the map obtained when we apply the functorial factorizations in $\cC$ to the horizontal arrows in the diagram above:
$$\xymatrix{ E(\alpha(b))\ar[d]\ar[r]   & H_f(\alpha(b))\ar[d]\ar[r]   & \lim_{A_{\alpha(b)}^{m-1}}H_f\times_{\lim_{A_{\alpha(b)}^{m-1}}F} F(\alpha(b))\ar[d] \\
 K(b)\ar[r] & H_t(b)\ar[r] & \lim_{B_{b}^{n-1}}H_t\times_{\lim_{B_{b}^{n-1}}G} G(b).}$$

Similarly, there exist a unique $l\geq0$ s.t. $\alpha'(b)\in A^l\setminus A^{l-1}$, and we have an induced commutative diagram:
$$\xymatrix{ E(\alpha'(b))\ar[d]\ar[r]   &  \lim_{A_{\alpha'(b)}^{l-1}}H_f\times_{\lim_{A_{\alpha'(b)}^{l-1}}F} F(\alpha'(b))\ar[d] \\
 K(b)\ar[r] & \lim_{B_{b}^{n-1}}H_t\times_{\lim_{B_{b}^{n-1}}G} G(b).}$$
The map $\chi'_b:H_f(\alpha'(b))\to H_t(b)$ is the map obtained when we apply the functorial factorizations in $\cC$ to the horizontal arrows in the diagram above:
$$\xymatrix{ E(\alpha'(b))\ar[d]\ar[r]   & H_f(\alpha'(b))\ar[d]\ar[r]   & \lim_{A_{\alpha'(b)}^{l-1}}H_f\times_{\lim_{A_{\alpha'(b)}^{l-1}}F} F(\alpha'(b))\ar[d] \\
 K(b)\ar[r] & H_t(b)\ar[r] & \lim_{B_{b}^{n-1}}H_t\times_{\lim_{B_{b}^{n-1}}G} G(b).}$$

Since $\alpha'(b)\geq\alpha(b)$ we clearly have an induced commutative diagram:
$$\xymatrix{  & E(\alpha'(b))\ar[dl]\ar[dd]\ar[rr]   &   &   F(\alpha'(b))\ar[dl]\ar[dd]\\
          E(\alpha(b))\ar[dd]\ar[rr] &  & F(\alpha(b))\ar[dd] &           \\
            & \lim_{A_{\alpha'(b)}^{l-1}}H_f\ar[dl]^r\ar[rr]   &   &   \lim_{A_{\alpha'(b)}^{l-1}}F\ar[dl]^s\\
          \lim_{A_{\alpha(b)}^{m-1}}H_f\ar[rr] &  & \lim_{A_{\alpha(b)}^{m-1}}F}$$

Combining all the above and using the induction hypothesis and the assumptions of the lemma we get an induced commutative diagram:
$$\xymatrix{  & E(\alpha'(b))\ar[dl]\ar[dd]\ar[r]   &   E(\alpha(b))\ar[dll]\ar[dd]\\
          K(b)\ar[dd] &  &           \\
            &  \lim_{A_{\alpha'(b)}^{l-1}}H_f\times_{\lim_{A_{\alpha'(b)}^{l-1}}F} F(\alpha'(b))\ar[dl]\ar[r]   &   \lim_{A_{\alpha(b)}^{m-1}}H_f\times_{\lim_{A_{\alpha(b)}^{m-1}}F} F(\alpha(b))\ar[dll]\\
          \lim_{B_{b}^{n-1}}H_t\times_{\lim_{B_{b}^{n-1}}G} G(b) &   &  }$$

Applying the functorial factorizations in $\cC$ to the vertical arrows in the diagram above gives us the inductive step.
\end{proof}

We need to show that the morphism we have constructed in $\barr{Pro}(\cC^{\Delta^2})$ between the Reedy factorizations does not depend on the choice of representative $(\alpha,\Phi)$ to the morphism $f\to t$ in $\barr{Pro}(\cC^{\Delta^1})$. So let $(\alpha',\Phi')$ be another representative.

Thus, $\alpha':B\to A$ is a strictly increasing function, $\Phi'=(\phi',\psi')$ is a pair of morphisms in $\cC^B$ and we have a commutative diagram in $\cC^B$:
$$\xymatrix{{E\circ\alpha'}\ar[r]^{{f_{\alpha'}}}\ar[d]_{\phi'} & {F\circ\alpha'}\ar[d]^{\psi'}\\
            K \ar[r]^t    &   G.}$$

We apply the $\chi$-construction to this new datum and obtain a natural transformation: $\chi':H_f\circ\alpha'\to H_t$.

The pre morphisms $(\alpha,\Phi)$,$(\alpha',\Phi')$ both represent the same morphism $f\to t$ in $\barr{Pro}(\cC^{\Delta^1})$, so by Corollary ~\ref{c_directed} there exist a  pre morphism $({\alpha''},{\Phi''})$ from $f$ to $t$ s.t. $({\alpha''},{\Phi''})\geq(\alpha,\Phi),(\alpha',\Phi')$.

Thus, $\alpha'':B\to A$ is a strictly increasing function, $\Phi''=(\phi'',\psi'')$ is a pair of morphisms in $\cC^B$ and we have a commutative diagram in $\cC^B$:
$$\xymatrix{{E\circ\alpha''}\ar[r]^{{f_{\alpha''}}}\ar[d]_{\phi''} & {F\circ\alpha''}\ar[d]^{\psi''}\\
            K \ar[r]^t    &   G.}$$

We apply the $\chi$-construction to this new datum and obtain a natural transformation: $\chi'':H_f\circ\alpha''\to H_t$.

$({\alpha''},{\Phi''})\geq(\alpha,\Phi),(\alpha',\Phi')$ means that for every $b\in B$ we have: $\alpha'(b)\geq\alpha(b)$ and the following diagrams commute:

\[
\xymatrix{ & E(\alpha(b)) \ar[dr]^{\phi_b} &  &  &   E(\alpha'(b)) \ar[dr]^{\phi'_b} &  \\
E(\alpha''(b))\ar[ur]\ar[rr]^{\phi''_b} & & K(b) &  E(\alpha''(b))\ar[ur]\ar[rr]^{\phi''_b} & & K(b).}
\]

\[
\xymatrix{ & F(\alpha(b)) \ar[dr]^{\psi_b} &  &  &   F(\alpha'(b)) \ar[dr]^{\psi'_b} &  \\
F(\alpha''(b))\ar[ur]\ar[rr]^{\psi''_b} & & G(b) &  F(\alpha''(b))\ar[ur]\ar[rr]^{\psi''_b} & & G(b).}
\]

Thus, to get the desired result, it remains to show that for every $b\in B$ the following diagrams commute:

\[
\xymatrix{ & H_f(\alpha(b)) \ar[dr]^{\chi_b} &  &  &   H_f(\alpha'(b)) \ar[dr]^{\chi'_b} & \\
H_f(\alpha''(b))\ar[ur]\ar[rr]^{\chi''_b} & & H_t(b)  &  H_f(\alpha''(b))\ar[ur]\ar[rr]^{\chi''_b} & & H_t(b).}
\]
But this follows from Lemma ~\ref{l_reedy wd}.

It remains to verify that we have indeed defined a functor.

We first check that the identity goes to the identity.

Let $f$ be an object of $\barr{Pro}(\cC^{\Delta^1})$. Then $f:E^A\to F^A$ is a natural transformation between objects in $\barr{Pro}(\cC)$. Clearly $(\alpha,\Phi)=(\alpha,\phi,\psi)=(id_A,id_E,id_F)$ is a representative to the identity morphism $f\to f$ in $\barr{Pro}(\cC^{\Delta^1})$.

We now need to apply the $\chi$-construction to $(\alpha,\phi,\psi)$. It is not hard to verify that we obtain the identity natural transformation: $\chi=id_{H_f}:H_f\circ\alpha\to H_f$. Thus the result of applying the functor to the identity is the identity.

We now check that there is compatibility w.r.t. composition.

Let $f,t,r$ be objects of $\barr{Pro}(\cC^{\Delta^1})$. Then $f:E^A\to F^A,t:K^B\to G^B,r:L^C\to M^C$ are natural transformation between objects in $\barr{Pro}(\cC)$. Let $(\alpha,\Phi)=(\alpha,\phi,\psi)$ be a representative to a morphism $f\to t$ in $\barr{Pro}(\cC^{\Delta^1})$ and $(\beta,\Psi)=(\beta,\gamma,\delta)$ be a representative to a morphism $t\to r$ in $\barr{Pro}(\cC^{\Delta^1})$. Then $$(\alpha\beta,\Psi\Phi_{\beta})=(\alpha\beta,\gamma\phi_{\beta},\delta\psi_{\beta})$$
is a representative to the composition of the above morphisms in $\barr{Pro}(\cC^{\Delta^1})$.

We now apply the $\chi$-construction to $(\alpha,\phi,\psi)$, and get a natural transformation: $\chi:H_f\circ\alpha\to H_t$, and we apply the $\chi$-construction to $(\beta,\gamma,\delta)$, and get a natural transformation: $\epsilon:H_t\circ\beta\to H_r$.

It is not hard to verify that applying the $\chi$-construction to $(\alpha\beta,\gamma\phi_{\beta},\delta\psi_{\beta})$ yields the natural transformation: $\epsilon\chi_{\beta}:H_f\circ(\alpha\beta)\to H_r$. Thus, applying the $\chi$-construction to the composition $(\beta,\Psi)\circ(\alpha,\Phi)$ yields the composition of the pre morphisms which are the $\chi$-constructions of $(\beta,\Psi)$ and $(\alpha,\Phi)$, as desired.

\end{proof}

\begin{cor}\label{c_wfs}
If $Mor(\mcal{C}) = ^{func} M\circ N$ and $N\perp M$, then $(Lw^{\cong}(N),R(Sp^{\cong}(M)))$ is a functorial weak factorization system in $Pro(\cC)$.
\end{cor}
\begin{proof}
This follows from Theorem ~\ref{t_fact} and Proposition ~\ref{p_wfs}.
\end{proof}

Department of Mathematics, Hebrew University, Jerusalem

\emph{E-mail address}:
\texttt{ilanbarnea770@gmail.com}

Department of Mathematics, Hebrew University, Jerusalem

\emph{E-mail address}:
\texttt{tomer.schlank@gmail.com}

\end{document}